\documentclass[11pt]{article}
\usepackage[a4paper, margin=0.7in]{geometry}
\usepackage{amsthm}
\usepackage[utf8]{inputenc}
\usepackage[T1]{fontenc}
\usepackage[english]{babel}
\usepackage{ifpdf,newtxtext,newtxmath} 
\usepackage{array,graphicx,dcolumn,multirow,hevea,abstract,hanging,fancyhdr,float}

\usepackage{setspace}
\usepackage{natbib} 
\usepackage{booktabs} 

\usepackage{thmtools}
\declaretheorem{theorem}

\declaretheorem{proposition} 
\declaretheorem{lemma}

\title{Limit cycles and Integrability of a continuous system with a line of equilibrium points}

\author{Aram A. Abdulkareem$^1$, Azad I. Amen$^{2, 3, 4}$, Niazy H. Hussein$^5$}

\date{} 
\begin{document} 

\maketitle

\begin{center}
	$^1$Department of Mathematics, Faculty of Education, Soran University, Erbil-Soran, Iraq. aram.abdulkareem$@$soran.edu.iq\\
	$^2$Department of Mathematics, College of Basic Education, Salahaddin
	University-Erbil, Erbil, Iraq.\\
	$^3$Department of Mathematics, Basic Education College, Raparin University - Ranya, Iraq.\\
	$^4$Department of Mathematics, College of Science, Duhok University, Iraq.\\
	$^5$Department of Mathematics, Faculty of Science, Soran University, Erbil-Soran, Iraq.\\
\end{center}

\bigskip

\begin{abstract}
	We focus on a chaotic differential system in 3-dimension, including an absolute term and a line of equilibrium points. Which describes in the following
	$$\dot{x} =y\, , \, \dot{y} =-ax+yz \, , \, \dot{z} =b |y| -cxy-x^2 \, .$$
	This system has an implementation in electronic components. The first purpose of this paper is to provide sufficient conditions for the existence of a limit cycle bifurcating from the zero-Hopf equilibrium point located at the origin of the coordinates.
	The second aim is to study the integrability of each differential system, one defined in half–space $y\ge 0$ and the other in half–space $y<0$. We prove that these two systems have no polynomial, rational, or Darboux first integrals for any value of $a$, $b$, and $c$. Furthermore, we provide a formal series and an analytic first integral of these systems. We also classify Darboux polynomials and exponential factors.

\end{abstract}

\section{Introduction and the main results}

Recent decades have seen a significant increase in the study of piecewise smooth differential systems, mainly because this type of system offers more realistic models in many applications, such as those involving switched circuit modeling, some mechanical problems, and control theory. For more details, one can see [\cite{broucke2001structural}, \cite{bernardo2008piecewise}, \cite{simpson2010bifurcations}, and \cite{feckan2016poincare} ].
In \cite{pham2016gallery}, the authors listed eight chaotic systems. One of them, which has infinite equilibrium points, is given in the following:
\begin{equation} \label{1}
\dot{x} =y, \quad	
\dot{y}=-ax+yz,	\quad                   
\dot{z}  =b |y| -cxy-x^2 \, ,
\end{equation}
where $a, b$ and $c$ are real parameters. When working with differential equations, determining if a differential system exhibits chaos is important [\cite{barreira2020integrability}]. It's worth mentioning that the absolute–value function is a possible nonlinear option for constructing chaotic systems with hidden attractors [\cite{li2015constructing}, \cite{munmuangsaen2015simple}]. The authors in [\cite{pham2019chaotic}] have investigated the dynamics of a system (\ref{1}) and observed its chaotic attractors and multistability depending on initial conditions.

The limit cycles play a significant role in the dynamical systems when they exist. The limit cycles of a piecewise differential system are a highly challenging problem to analyze. The authors in [\cite{llibre2020crossing}, \cite{llibre2018periodic}] studied the limit cycles of the piecewise differential systems, linear or nonlinear, using the first integrals. On the other hand, the authors in [\cite{kassa2021limit}] studied the limit cycles bifurcating from a zero-Hopf equilibrium point using the averaging theory for Lipschitz differential systems.

Our first aim is to extend the dynamical features of system (\ref{1}) by showing that it can exhibit a zero-Hopf equilibrium point for appropriate parameter values for which one limit cycle can bifurcate from the origin. The following is the main result, which relates to limit cycles.

\begin{theorem} \label{Z-H-THe}
Consider the differential system (\ref{1}) with $b=\epsilon \beta$, and $\epsilon \gg 0$ sufficiently small. System (\ref{1}) has one unstable limit cycle $(x(t, \epsilon), y(t, \epsilon), z(t, \epsilon))$ which bifurcates from a non-isolated zero-Hopf equilibrium point located at the origin, if $a, \beta>0$.
\end{theorem}
The proof of Theorem \ref{Z-H-THe} is given in Section \ref{Sec Z-H}, which uses the averaging theory; this method is described in the appendix (Theorem \ref{AvThe}).

The second objective of this paper is to study the integrability of system (\ref{1}). It is important to note that, currently, there is no sophisticated method for dealing with the integrability of non-smooth vector fields, especially if the system is defined on non-compact manifolds. Moreover, the classical method of integrability for the smooth system cannot be used directly for system (\ref{1}). The absolute value term $|y|$ is regarded in system (\ref{1}) as a piecewise function, which is defined as:
\begin{equation}
	|y|=
	\begin{cases}
		y & \text{if } y\ge 0,\\
		-y & \text{if } y<0 \, .
	\end{cases}
\end{equation}
The non-smooth differential system (\ref{1}) is formed by the following two smooth differential systems:

Considering $y \ge 0$, then system (\ref{1}) becomes
\begin{equation} \label{1ybig}
	\dot{x} =y , \quad 	
	\dot{y} =-ax+yz , \quad               
	\dot{z} =b y -cxy-x^2 \, .
\end{equation}

Also considering $y < 0$, then system (\ref{1}) becomes
\begin{equation} \label{1ysmall}
	\dot{x} =y , \quad 	
	\dot{y} =-ax+yz , \quad               
	\dot{z} =-by -cxy-x^2 \, .
\end{equation}
Consequently, the classical method can be used to study the integrability of systems (\ref{1ybig}) and (\ref{1ysmall}), respectively. More precisely, we use the Darboux theory of integrability to report the existence or non-existence of Darboux polynomials, exponential factors, Darboux first integrals, polynomial, and rational first integrals. Moreover, we provide the existence of the formal and analytic first integral of systems (\ref{1ybig}) and (\ref{1ysmall}), respectively.

The second result, which is related to the integrability of system (\ref{1ybig}) is summarized in the following two theorems.
\begin{theorem}\label{INT-S1-big}
The following statements hold for system (\ref{1ybig})
\begin{enumerate}
	\item System (\ref{1ybig}) has no polynomial first integrals.
	\item System (\ref{1ybig}) has a unique irreducible Darboux polynomial represented as $y$, with cofactor $z$ if and only if $a=0$.
	\item System (\ref{1ybig}) has no rational first integrals.
	\item System (\ref{1ybig}) has a single exponential factor $e^x$ with cofactor $y$.
	\item System (\ref{1ybig}) have no first integrals of the Darboux type.
\end{enumerate}
\end{theorem}

\begin{theorem}\label{THformal}
	System (\ref{1ybig}) has a formal first integral in the neighborhood of equilibrium points $E(0, 0, z_0)$ for following cases: 
\begin{enumerate}
	\item $z_0>4a$, where $a>0$. 
	\item $a<0$ and $z_0 \in \mathbb{R} \setminus \{0\}$.
\end{enumerate}
	Moreover, System (\ref{1ybig}) has an analytic first integral in the neighborhood of equilibrium points $E(0,0,z_0)$ for $z_0 \in (-2\sqrt{a} , 2\sqrt{a})\setminus \{0\}$ and $a>0$.
\end{theorem}

The proof of Theorem \ref{INT-S1-big} and \ref{THformal} is given in Section \ref{Sec INT}. It is important to note that the results of Theorem \ref{INT-S1-big} and \ref{THformal} are also satisfied system (\ref{1ysmall}), and their proof is mainly similar. So, we only consider the results on system (\ref{1ybig}).

\section{Integrability for smooth system}

In this section, we give some results on the Darboux theory of integrability we use throughout this investigation. The researchers [\cite{dumortier2006qualitative},  \cite{christopher2004darboux}, \cite{llibre2015darboux},  \cite{barreira2015integrability}, \cite{llibre2012darboux}] have contributed into great details about the Darboux theory of this type of system.

The vector field corresponding to system (\ref{1ybig}) is
\begin{equation}
\mathcal{X}=y\frac{\partial}{\partial x} +(-ax+yz) \frac{\partial}{\partial y} +(b y -cxy-x^2) \frac{\partial}{\partial z}\, .
\end{equation}

We called $f \in \mathbb{C}[x, y, z]$ is a Darboux polynomial of the vector field $\mathcal{X}$ if there exists a polynomial (which is the cofactor of $f$) $K \in \mathbb{C}[x, y, z]$ such that $\mathcal{X}(f)=K f $. Where $\mathbb{C}$ denotes the ring of all complex polynomials in the variables $x, y, z$. Here, the cofactor is a polynomial of degree at most $1$. If $f(x,y,z)$ is a Darboux polynomial of system (\ref{1ybig}), then $f(x,y,z)=0$ is an invariant algebraic surface in $\mathbb{R}^3$, i.e. if an orbit has the point on the surface $f(x, y, z) = 0$, the entire orbit is contained on it.

A Darboux polynomial $f(x,y,z) \in \mathbb{C}[x,y,z]$ with a zero cofactor is defined as a polynomial first integral of the system (\ref{1ybig}). When $f$ is a rational function that is not a polynomial, we say that $f$ is a rational first integral. If $f$ is a formal series, then $f$ is called a formal first integral. Moreover, if $f$ is an analytic
function, we say that $f$ is an analytic first integral.

A non-constant function $E= e^{(g/h)}$ is an exponential factor of system (\ref{1ybig}) with $g$ and $h$ are coprime polynomials in  $\mathbb{C}[x, y, z]$, if $\mathcal{X} (E)=L E$. Where the polynomial $L$ is a cofactor of $E$ with a maximum degree of one, see [\cite{christopher2007multiplicity}].

If a first integral $f$ of system (\ref{1ybig}) is of the following form,
\begin{equation}\label{H dar type}
	H=f_1^{\lambda_1} f_2^{\lambda_2} \ldots f_p^{\lambda_p} \, E_1^{\mu_1} E_2^{\mu_2} \ldots E_q^{\mu_q},
\end{equation}
it is called a Darboux type or a Darboux first integral, where $f_j$ is a Darboux polynomial, $E_k$ is an exponential factor, $\lambda_j, \mu_k \in \mathbb{C}$ for $j=1,\ldots,p$, $k=1,\ldots,q$. The following theorem gives the condition for enough number of Darboux polynomials, and exponential factors need to have a Darboux first integral.
\begin{theorem} \label{thDarType}
Assume that system (\ref{1ybig}) admits $p$ invariant algebraic surfaces $f_i = 0$ with cofactors $K_i$ for $i = 1,\ldots, p$, and $q$ exponential factors $E_j = e^{(g_j/h_j)}$ with cofactors $L_j$ for $j=1,\ldots, q$. Then there exist $\lambda _i$ and $\mu_j$ in $\mathbb{C}$, which are not all zero, such that
	\begin{equation}\label{EqDarType}
		\sum_{i=1}^{p} \lambda_i K_i +\sum_{j=1}^{q} \mu_j L_i = 0
	\end{equation}
	if and only if system (\ref{1ybig}) has Darboux first integrl $H$ of the form (\ref{H dar type}).
\end{theorem}

The following two results are associated with the existence of Darboux polynomials. For proof, see [\cite{christopher2007multiplicity}, \cite{llibre2009darbouxa}, \cite{llibre2012darboux}].

\begin{lemma} \label{THration1}
	For system (\ref{1ybig}), the existence of a rational first integral indicates the presence of either a polynomial first integral or two Darboux polynomials with a non-zero cofactor.
\end{lemma}

\begin{proposition}\label{propo exp}
	Both of the following statements hold.
	\begin{enumerate}
		\item If $E=e^{(g/h)}$ is an exponential factor for the polynomial system (\ref{1ybig}) and $h$ is not a constant polynomial, then $h=0$ is an invariant algebraic surface.
		\item Eventually, $e^g$ can be an exponential factor, coming from the multiplicity of the infinity.
	\end{enumerate}
\end{proposition}

To prove Theorem \ref{INT-S1-big} we use the weight-homogeneous polynomials. Where the weight-homogeneous polynomial has been used extensively in several standard systems  [\cite{llibre2002invariant},  \cite{llibre2012integrability}, \cite{jalal2020darboux}].

A polynomial $h(x)$ with $x\in \mathbb{R}^n$ is said to be weight homogeneous if there exist $r = (r_1,\ldots, r_n)\in \mathbb{N}^n$ and $m\in \mathbb{N}$ such that  $h(\mu^{r_1} x_1,\ldots,\mu^{r_n} x_n) = \mu^m h(x)$ for all $\mu >0$. $r$ is called the weight exponent, and $m$ is called the weight degree.

We use the following two result from [\cite{zhang2003local}, \cite{zhang2017note}].
Firstly, we consider an analytic differential system
\begin{eqnarray} \label{formaleqq}
	\dot{\mathbf{x}}=\mathbf{f(x)} \, ,\quad \mathbf{x}= (x_1, x_2, x_3) \in \mathbb{R}^3
\end{eqnarray}
where $\mathbf{f(x)}$ is a vector-valued function satisfying $\mathbf{f(0)=0}$. We denote by $A=\partial \mathbf{f} / \partial \mathbf{x}$ the Jacobian matrix of system (\ref{formaleqq}) at $\mathbf{x = 0}$.

\begin{theorem}\label{Theoana1}
	Assume that the eigenvalues $\lambda_1, \lambda_2, \lambda_3$ of the Jacobian matrix $A$ satisfy $\lambda_1=0$ and $k_2 \lambda_2 +k_3 \lambda_3 \ne 0$ for any $k_2, k_3, \in \mathbb{Z}^+ \cup {0}$ with $k_2 +k_3 \ge 1$.
	System (\ref{formaleqq}) has a formal first integral in the neighborhood of $\mathbf{x=0}$ if and only if the equilibrium point $\mathbf{x=0}$ is not isolated. In particular, if the equilibrium point $\mathbf{x=0}$ is isolated, then system (\ref{formaleqq}) has no analytic first integral in a neighborhood of $\mathbf{x=0}$.
\end{theorem}

\begin{theorem}\label{Theoana2}
	For the local analytic differential system (\ref{formaleqq}), assume that $\lambda_1 = 0$, and	$\lambda_2, \lambda_3$ either all have positive real parts or all have negative real parts. Then the system (\ref{formaleqq}) has an analytic first integral in a neighborhood of $\mathbf{x= 0}$ if and only if the equilibrium point $\mathbf{x= 0}$ is not isolated.
\end{theorem}

\section{Proof of Theorem \ref{Z-H-THe}}\label{Sec Z-H}
\begin{proof}[Proof of Theorem \ref{Z-H-THe}]
	Consider $b=\epsilon \beta$, and $\epsilon \geq 0$ in system (\ref{1}). If $\epsilon =0$, the Jacobian matrix of system (\ref{1}) at the equilibrium point $E(0,0,z_0)$  is
$$\left[ \begin {array}{ccc} 
	0& 1& 0\\
	-a &z_0 &0\\
	0 & 0 & 0
\end {array} \right]$$
Sitting $z_0=0$, Then the above matrix has on zero eigenvalue and $\pm \sqrt{-a}$ eigenvalues. Assuming that $a>0$, then system (\ref{1}) has a non-isolated zero-Hopf equilibrium localizing at the origin of the coordinates with eigenvalues $\lambda_1=0$, $\lambda_{2,3}=\pm i\sqrt{a}$.

We will use the averaging theory to estimate the limit cycle, which bifurcates from the origin. Before this, we must formulate system (\ref{1}) into normal form (\ref{Av1}) in the appendix. 
We start by rescaling the variables $(x,y,z)$ to $(\epsilon X,\epsilon Y,\epsilon Z)$, this gives us
\begin{equation} \label{z2}
\dot{X} =Y \, , \quad	
\dot{Y} =\epsilon YZ- a X \, , \quad			
\dot{Z} =\epsilon \beta \, |Y|-\epsilon c XY-\epsilon X^2\, .
\end{equation}
Now, using the change of variables $(X,Y,Z)\to (u,v,w)$ by
$$\left[ \begin {array}{c} X\\ Y\\ Z 
\end {array} \right] =
\left[ \begin {array}{ccc} 
	1& 0& 0\\
	0&-\sqrt{a} &0\\
	0& 0 & 1
\end {array} \right]
\left[ \begin {array}{c} u\\ v\\ w 
\end {array} \right] \, .$$
The following differential system is obtained
\begin{equation} \label{z3}
\dot{u} =-\sqrt{a} \, v \, , \quad	
\dot{v} =\sqrt{a} \, u+\epsilon \, vw \, , \quad
\dot{w} =\epsilon \left( \sqrt{a} \, \beta \, |v| + c\sqrt{a} \, uv- u^2\right) \, .
\end{equation}
Also the cylindrical coordinates $(r,\theta,w)$ are defined as $u= r  \cos (\theta)$ and  $v=r \sin (\theta)$. System (\ref{z3}) becomes
\begin{equation} \label{z4}
\dot{r} =\epsilon \, w\, r \, \sin^2 \theta  \, , \quad 	
\dot{\theta} =\sqrt {a}+\epsilon\, w\, \cos\theta \, \sin\theta \, , \quad
\dot{w} =\epsilon\left( \sqrt{a} \, \beta |r\sin \theta| + r^2 \cos \theta \left( c\sqrt{a} \sin \theta - \cos \theta \right)  \right) .
\end{equation}
We utilize $\theta$ as a new independent variable. System (\ref{z4}) then becomes
\begin{equation}\label{zav}
\left( \frac{dr}{d\theta} ,\frac{dw}{d\theta}\right) =\left(F_1(r,\theta,w) , F_2(r,\theta,w) \right) ,
\end{equation}
where
\begin{equation*}
F_1(r,\theta,w) =\frac{1}{\sqrt{a}} \, r\, w \, \sin^2\theta \, , \quad
F_2(r,\theta,w) =\beta\, |r| |\sin\theta| +c\, r^2 \, \cos\theta \sin \theta- \frac{r^2 \cos^2\theta}{\sqrt{a}} .
\end{equation*}
This gives the averaged function
\begin{equation}
F_{10}(r,w) =\frac{1}{2\pi}\int_{0}^{2\pi} F_1(r,\theta,w) d\theta=\frac{r\,w}{2\sqrt{a}} \, , \quad
F_{20}(r,w) =\frac{1}{2\pi}\int_{0}^{2\pi} F_2(r,\theta,w) d\theta =\frac{4\beta\sqrt{a}|r| - \pi r^2}{2\pi \sqrt{a}}\, .
\end{equation}
The system $F_{10}(r,w)=F_{20}(r,w)=0$ has a unique solution $(r^*,w^*)$ with $r^*>0$ for $\beta>0$, namely
$$(r^*,w^*)=\left(\frac{4\beta \, \sqrt{a}}{\pi} , 0 \right) .$$
The determinant of the Jacobian matrix at $(r^*,w^*)$ is
\begin{equation*}
\det \left(\frac{\partial F_0(r,w)}{(\partial r, \partial w)} \right)_{(r^*,w^*)} =-\frac{4 \beta^2}{\pi^2} \neq 0 \, , \quad where \, \beta \neq 0
\end{equation*}
This gives that, system (\ref{zav}) has a limit cycle $(r(\theta,\epsilon),w(\theta,\epsilon))$ satisfying $(r(\theta,\epsilon),w(\theta,\epsilon))\to (r^*,w^*) + O(\epsilon)$. Going back to system (\ref{1}), the limit cycle of system (\ref{1}) is $(x(t,\epsilon),y(t,\epsilon),z(t,\epsilon))$ fulfills
$$(x(0,\epsilon),y(0,\epsilon),z(0,\epsilon)) \to (\epsilon r^*,0,\epsilon w^*) + O(\epsilon ^2) \to \epsilon((4\beta \sqrt{a})/\pi,0,0)+O(\epsilon^2).$$
When $\epsilon \to 0$, this limit cycle tends to the equilibrium point located at the origin of the coordinates. Since, the eigenvalues of the matrix $\left(\frac{\partial F_0(r,w)}{(\partial r, \partial w)} \right) $ evaluated at $\left(\frac{4\beta \, \sqrt{a}}{\pi} , 0 \right)$ are $\pm \frac{2 \beta}{\pi}$, by Theorem \ref{AvThe}, it follows that the limit cycle defining by $\left(\frac{4\beta \, \sqrt{a}}{\pi} , 0 \right)$ is unstable.
\end{proof}

Here, we will exhibit an example showing that one unstable limit cycle is born at the origin when $\epsilon \to 0$ of Theorem \ref{Z-H-THe}. we consider $a=2.5$, $c=9$ and $b=\epsilon \beta$, where $\beta=2$ and $\epsilon=0.0001$. We obtain one positive solution for $r^*>0$, with the initial condition $(0.0004026336968,0,0)$ shown in Figure \ref{fig:limitcycle}.
\begin{figure}
	\begin{center}
		\includegraphics[width=1 \textwidth]{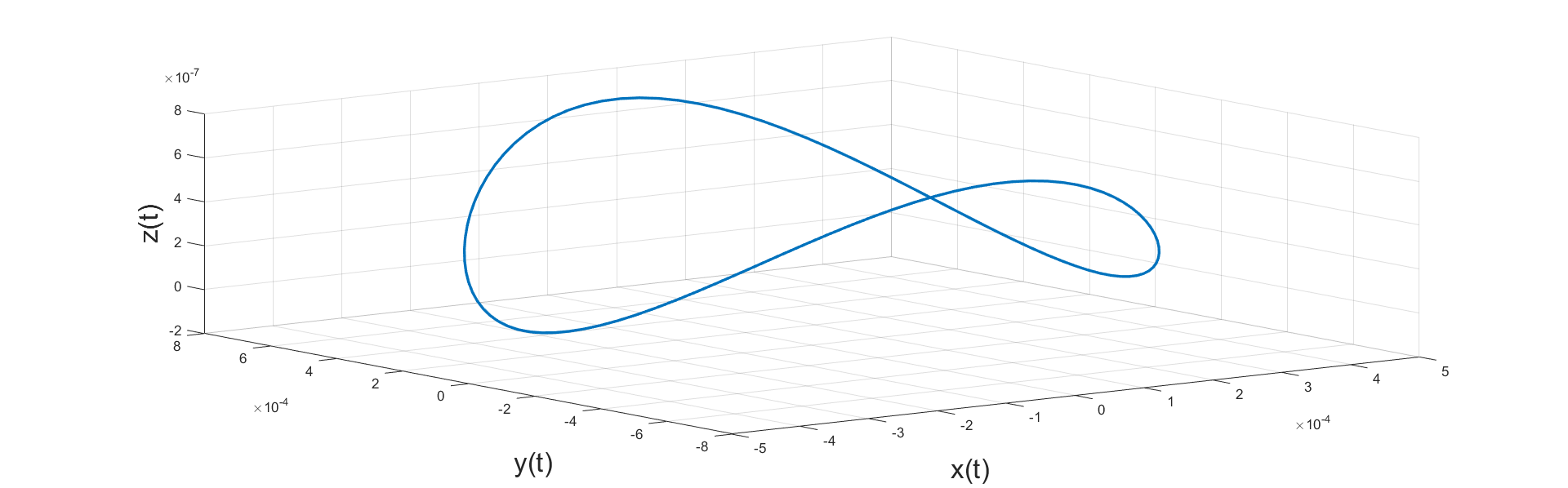}
		\caption{Limit cycle bifurcting from origin in system \ref{1}, for $a=2.5, c=9, b=\epsilon \beta, \beta=2$, $\epsilon=0.0001$ with initial condition:$(0.0004026336968,0,0)$.}
		\label{fig:limitcycle}
	\end{center}
\end{figure}

\section{Proof of Theorems \ref{INT-S1-big} and \ref{THformal}}\label{Sec INT}
\begin{proof}[Proof of Therem \ref{INT-S1-big} (1)]
Let $H=H(x,y,z)$ be a polynomial first integral of system (\ref{1ybig}). Then, it satisfies
\begin{equation}\label{i1}
	y\frac{\partial H}{\partial x}+(-ax+y z)\frac{\partial H}{\partial y}+(by-cxy-x^2)\frac{\partial H}{\partial z}=0.
\end{equation}
We can write $H$ in the form $H(x,y,z)= \sum_{i=0}^{n} H_i(x,y) z^i$, where each $H_i$ is a polynomial in the variables $x,y$. Now, computing the coefficient in (\ref{i1}) of $z^{n+1}$, we obtain
$y\frac{\partial}{\partial y} H_n (x,y)=0$. That is $H_n (x,y)=F_n (x)$, where $F_n(x)$ is a polynomial of the variable $x$. Computing also the coefficient in (\ref{i1}) of $z^{n}$, we obtain
$y\frac{\partial}{\partial y} H_{n-1} (x,y)+y\frac{d}{dx}F_n(x)=0.$
That is
$$H_{n-1} (x,y)=-y \frac{d}{dx}F_n (x)+F_{n-1} (x),$$
where $F_{n-1}(x)$ is a polynomial of the variable $x$.  Computing the coefficient in (\ref{i1}) of $z^{n-1}$, we get
$$y\frac{\partial}{\partial y} H_{n-2} (x,y)+ax\frac{d}{dx}F_n(x)-y^2\frac{d^2}{dx^2}F_n(x)+y\frac{d}{dx}F_{n-1}(x)+(b-cx)nyF_n(X)-nx^2F_n(x)=0.$$
Solving the above equation with respect to $H_{n-2}(x,y)$, we obtain
$$H_{n-2}(x,y)=\frac{1}{2}y^2\frac{d^2F_n(x)}{dx^2}-y\frac{dF_{n-1}(x)}{dx}-(b-cx)nyF_n(X)+\left(nx^2F_n(x)-ax\frac{dF_{n}(x)}{dx} \right)\ln(y) +F_{n-2}(x).$$
Since $H_{n-2}$ is a polynomial, then we must have $nx^2F_n(x)-ax\frac{dF_{n}(x)}{dx}  =0$.
This gives
\begin{equation*}
	F_{n}(x)=C_n e^{\frac{n x^2}{2 a}},
\end{equation*}
where $C_n$ is a constant. Since $F_{n}(x)$ is a polynomial, this implies that $n=0$. Therefore, for $n\geq 1$, then $H_n=0$. Let $H=H_0(x,y)$, we substitute it in (\ref{i1}), and computing the coefficient in (\ref{i1}) of $z^i$ for $i=0,1$. We obtain

for $i=1$: $y\frac{\partial}{\partial y} H_{0}(x,y)=0$, that is $H_0(x,y)=F_0(x)$, where $F_0(x)$ is a polynomial of the variable $x$.
	
for $i=0$: $y\frac{\partial}{\partial y} F_{0}(x)=0$, that is $F_0(x)=C$, where $C$ is a constant.\\
Hence, $H(x,y,z)=C$. So, system (\ref{1ybig}) has no polynomial first integrals.
\end{proof}

Now, we will start the proof of Theorem \ref{INT-S1-big}(2). We recall that a Darboux polynomial of the system (\ref{1ybig}) is a non-constant polynomial $f\in \mathbb{C}[x,y,z]$ such that
\begin{equation}\label{dar}
	y\frac{\partial f}{\partial x}+(-ax+yz)\frac{\partial f}{\partial y}+(by-cxy-x^2)\frac{\partial f}{\partial z}=Kf \,.
\end{equation}
For some polynomial $K=k_0+k_1 x+k_2 y+k_3 z$, where $k_i\in \mathbb{C}$ for $i=0,1,2,3$. Firstly, we want to show that $k_0=k_1=k_2=0$
\begin{lemma}\label{lem-k2}
	$k_2=0$.
\end{lemma}
\begin{proof}
We write $f$ as the form $f(x,y,z)=\sum_{i=0}^{n}f_i(x,z) y^i$, where each $f_i$ is a polynomial of the variables $x,z$. Computing the coefficient in (\ref{dar}) of term $y^{n+1}$. We have
$$\frac{\partial}{\partial x}f_n(x,z)+(b-cx)\frac{\partial}{\partial z}f_n(x,z)-k_2 f_n(x,z)=0 . $$
Solving the above equation, we obtain
$f_n (x,z)=G_n (\frac{1}{2} cx^2-bx+z)  e^{k_2 x}$, where $G_n$ is an arbitrary polynomial of the variable $x$ and $z$. Since $f_n$ is a polynomial then must be $G_n=0$ or $k_2=0$. Now, suppose that $G_n=0$ and $k_2\neq 0$. Consequently, we obtain $f_n=0$ for $n\ge 1$. Thus $f=f_0 (x,z)$. We substitute it in (\ref{dar}) and compute the terms of $y$, we obtain
$$\frac{\partial}{\partial x}f_0(x,z)+(b-cx)\frac{\partial}{\partial z}f_0(x,z)-k_2 f_0(x,z)=0.$$
Then, $f_0 (x,z)=G_0 \left( \frac{1}{2} cx^2-bx+z \right)  e^{k_2 x}$ for an arbitrary function $G_0$ of the variables $x$ and $z$. If $k_2 \ne 0$, we get a contradiction with the fact that $f_0$ must be non-constant polynomial. Hence, $k_2=0$. This complete the proof of the lemma.
\end{proof}

\begin{lemma}\label{lem-k0k1}
	$k_0=k_1=0$ and $k_3 \in \mathbb{Z}^+$.
\end{lemma}
\begin{proof}
	From Lemma \ref{lem-k2}, we can consider $k_2=0$. We write $f=\sum_{i=1}^n f_i (x,y) z^i$, where each $f_i$ is a polynomial of the variables $x,y$. Computing the coefficient in (\ref{dar}) of the terms $z^{n+1}$, we obtain
	$$y \frac{\partial f_n(x,y)}{\partial y}-k_3 f_n(x,y)=0 .$$
Solving the above equation, we obtain $f_n(x,y)=F_n(x) y^{k_3}$.
Since $f_n (x,y)$ is a polynomial, it is possible $k_3\in \mathbb{Z}^+\cup\{0\}$. Computing also the terms of $z^n$ in (\ref{dar}), we obtain
\begin{equation*}
y \frac{\partial f_{n-1} (x,y)}{\partial y}-k_3 f_{n-1}(x,y) + y^{k_3+1}  \frac{dF_n(x)}{dx}-(k_0+k_1 x) F_n(x) y^{k_3}-ak_3 xy^{k_3-1} F_n(x)=0 .
\end{equation*}
Solving the above equation with respect to $f_{n-1}$, we obtain
\begin{equation*}
f_{n-1}(x,y)=\left( (k_0+k_1 x) \ln(y) F_n(x) -y\frac{dF_n(x)}{dx} - \frac{ak_3 x F_n(x)}{y} +F_{n-1}(x) \right) y^{k_3},
\end{equation*}
where $F_{n-1}(x)$ is a polynomial of $x$. Since $f_{n-1}$ is a polynomial, then $F_n (x)=0$ or $k_0=k_1=0$ and $k_3\in \mathbb{Z}^+$. Now, suppose that $F_n (x)=0$, $k_0\neq0$ and $k_1\neq0$. This gives that $f_n (x,y)=0$ for $n\ge 1$, this implies that $f=f_0 (x,y)$.
Then from (\ref{dar}) and computing the terms of $z^i$ for $i=1,0$, we have
\begin{equation}\label{dz1}
y \frac{\partial f_{0} (x,y)}{\partial y}-k_3 f_{0}(x,y)=0,
\end{equation}
\begin{equation}\label{dz0}
y\frac{\partial f_0(x,y)}{\partial x} - ax \frac{\partial f_0(x,y)}{\partial y} -(k_0+k_1 x)f_0(x,y) = 0.
\end{equation}
Solving the equation (\ref{dz1}), we obtain $f_0 (x,y)=F_0 (x)  y^{k_3}$. By substituting $f_0$ in equation (\ref{dz0}), we obtain 
\begin{equation*}
y^{k_3+1} \frac{d F_0(x)}{d x} - a k_3 x y^{k_3-1} -(k_0+k_1 x)F_0(x)=0.	
\end{equation*}
That is
$F_0(x)=Ce^{\frac{ak_3 x^2+k_1 x^2 y+2k_0 xy}{2y^2}}$,
where $C$ is an arbitrary constant. This contradict with the fact that $f_0$ is a polynomial of the variable $x$ and $y$. This implies that $k_0=k_1=0$, and $k_3\in \mathbb{Z}^+$. This completes the proof of the lemma. 
\end{proof}
Next, we demonstrate the proof of Theorem  \ref{INT-S1-big} (2).

\begin{proof}[Proof of Theorem \ref{INT-S1-big} (2)]
From Lemmas \ref{lem-k2} and \ref{lem-k0k1}, we can write $K$ as $k_3 z$, $k_3 \in \mathbb{Z}^+$. Now, We choose the change of variables $x=X$, $y=Y$, $z=\lambda^{-1}Z$, $t=\lambda T$. Then, system (\ref{1ybig}) becomes
\begin{equation} \label{w1}
\dot{X} =\lambda Y \, ,\quad
\dot{Y} =-a\lambda X+YZ \, , \quad
\dot{Z} =b \lambda^2 Y- c\lambda^2 XY - \lambda^2 X^2 , 
\end{equation}
where the dot represents derivative with respect to the variable $T$. Set $F(X,Y,Z)=\lambda^n f(X, Y, \lambda^{-1} Z)$ and $K(X, Y, Z) = \lambda K(X, Y, \lambda^{-1}Z)=k_3 Z$, where $n$ denotes the highest weight degree in the weight homogeneous components of $f$ in the variables $(x, y, z)$ with weight degree $(0, 0, 1)$.
	
Assume that $F(\lambda,X,Y,Z)=\sum_{i=0}^{n}\lambda^i F_{n-i}(X,Y,Z)$, where $F_{j}$ is a weight homogeneous polynomial of the variables $X$, $Y$ and $Z$ with weight degree $j$ for $j=0, 1,\ldots, n$. 
From the definition of Darboux polynomial, we have
\begin{equation}\label{w2}
\lambda Y \sum_{i=0}^{n}\lambda^i \frac{\partial F_{n-i}}{\partial X}
+(-a\lambda X+YZ) \sum_{i=0}^{n}\lambda^i \frac{\partial F_{n-i}}{\partial Y}                     
+(b \lambda^2 Y - c\lambda^2 XY - \lambda^2 X^2) \sum_{i=0}^{n}\lambda^i \frac{\partial F_{n-i}}{\partial Z} 	=(k_3 Z) \sum_{i=0}^{n}\lambda^i F_{n-i} \, .
\end{equation}
Computing the terms with $\lambda^0$ in (\ref{w2}), we get
$YZ \frac {\partial F_n(X,Y,Z)}{\partial Y} - k_3 Z F_{n} (X,Y,Z)=0$. That is $F_n(X,Y,Z)=H_n(X,Z) Y^{k_3}$, where $H_n(X,Z)$ is a polynomial of the variable $X$ and $Z$. Since $F_n$ is a homogeneous polynomial of weight degree $n$, we can assume that $H_n=W_n(X)Z^n$ with $W_n$ is an arbitrary polynomial of the variable $X$.
	
Computing the coefficient of $\lambda^1$ in (\ref{w2}), we obtain
\begin{equation*}
	YZ \frac {\partial F_{n-1}(X,Y,Z)}{\partial Y} - k_3 Z F_{n-1} (X,Y,Z)+Z^n Y^{k_3+1}\frac{d W_n(X)}{dX}-ak_3XZ^nY^{k_3-1} W_n(X)=0 .
\end{equation*}
Solving the above equation, we obtain
\begin{equation*}
	F_{n-1}(X,Y,Z)=-Z^{n-1}\left( a k_3 X W_n(X)Y^{k_3-1}+Y^{k_3+1}\frac{d W_n(X)}{dX}\right) +H_{n-1}(X,Z) Y^{k_3}.
\end{equation*}
Since $F_{n-1}$ is a homogeneous polynomial of weight degree $n-1$, we can assume that  $H_{n-1}=W_{n-1}(X)Z^{n-1}$, where $W_{n-1}$ is an arbitrary polynomial of the variable $X$.

Computing the coefficient of $\lambda^2$ in (\ref{w2}), we obtain
\begin{equation*}
\begin{split}
&YZ\frac {\partial F_{n-2}(X,Y,Z)}{\partial Y} -k_3ZF_{n-2}(X,Y,Z) +\left( (nb-nXc) W_n(X) + \left( \frac {d W_{n-1}(X)}{dX} -ak_3XW_{n-1}(X)  \right) \right) Z^{n-1} Y^{k_3+1} \\
& +\left(  (-n{X}^2-ak_3) W_n(X) + \left( aX{\frac {d W_n(X)}{dX}}  - Y^2{\frac {d^2 W_n(X)}{dX^2}}  \right) \right)  Z^{n-1} Y^{k_3}  + ({k_3}-1) a^2k_3W_n(X) X^2 Z^{n-1} {Y}^{k_3-2}= 0 \,.
	\end{split}
\end{equation*}
That is
\begin{equation}\label{eq Fn-2}
	\begin{split}
F_{n-2}(X,Y,Z)=& H_{n-2}(X,Z) Y^{k_3} +Z^{n-2}Y^{k_3} \ln (Y)  \left( ( a{k_3}+nX^2) W_n(X) -  aX\frac {d W_n(X)}{dX}  \right)  \\
&  +{Z}^{n-2} \left(  Y^{k_3-2} \left( \frac{1}{2} (k_3-1)k_3 a^2 X^2 W_n(X) \right) -  a k_3 X Y^{k_3-1}W_{n-1}(X)  \right) \\
&+{Z}^{n-2} Y^{k_3+1} \left( n(cX - b)W_n(X)  + \frac{1}{2} Y\frac {d^2 W_n(X)}{dX^2}  - \frac{d W_{n-1}(X)}{dX} \right) ,
	\end{split}
\end{equation}
where $H_{n-2}(X,Z)$ is a polynomial of the variables $X$ and $Z$. For $F_{n-2}(X,Y,Z)$ to be a polynomial, we consider two cases;

\textbf{Case 1.} if $n=a=0$ and $W_n(X)\ne 0$, then equation (\ref{eq Fn-2}) becomes
\begin{equation*}
F_{n-2}(X,Y,Z)=\frac{Y^{k_3+2}}{2 Z^2}\left( \frac{d^2 W_n(X)}{d X^2}- \frac{d W_{n-1}(X)}{d X} \right) + H_{n-2}(X,Z) Y^{k_3}.
\end{equation*}
Which is impossible, because $F_{n-2}$ is a polynomial.
	
\textbf{Case 2.} If $W_n(X)=0$, we can infer from equation (\ref{eq Fn-2}) that
\begin{equation*}
F_{n-2}(X,Y,Z)=-Z^{n-2}\left( a k_3 X Y^{k_3-1} W_{n-1}(X) + Y^{k_3+1}\frac{d W_{n-1}(X)}{d X} \right) + H_{n-2}(X,Z) Y^{k_3}.
\end{equation*}
Since $F_{n-2}$ is a polynomial of weight degree $n-2$, we can assume that $H_{n-2}(X,Z)=W_{n-2}(X) Z^{n-2}$. 

Similarly to the above process. By computing the coefficient of $\lambda^i$ for $i=3\ldots n-1$ in (\ref{w2}), we can derive $W_{n-1}(X)=\ldots=W_{3}=0$ and 
\begin{equation*}
\begin{split}
&F_{n-3}(X,Y,Z)=-Z^{n-3}\left( a k_3 X Y^{k_3-1} W_{n-2}(X) + Y^{k_3+1}\frac{d W_{n-2}(X)}{d X} \right) + H_{n-3}(X,Z) Y^{k_3},	\\
&. \\
&. \\
&F_{1}(X,Y,Z)=-Z \left( a k_3 X Y^{k_3-1} W_{2}(X) + Y^{k_3+1}\frac{d W_{2}(X)}{d X} \right) + H_{1}(X,Z) Y^{k_3} .
\end{split}
\end{equation*}
Since, $F_j$ is a polynomial of weight degree $j$, we can assume that $H_{j}(X,Z)=W_{j}(X) Z^{j}$, for $j=n-3, \ldots ,1$. 

Now, computing the coefficient of $\lambda^n$ in (\ref{w2}), we obtain
\begin{equation*}
\begin{split}
&YZ  {\frac {\partial F_0(X,Y,Z)}{\partial Y}}  - k_3 Z F_0(X,Y,Z) + ak_3Z\left( a(k_3-1)X^2Y^{k_3-2} W_2(X)  -  XY^{k_3-1} W_1(X) \right) \\
& +ZY^{k_3} \left( (a k_3 -2X^2)W_2(X)+aX  \frac {d W_2(X)}{dX}+2(b-cX) Y W_2(X)  +Y \frac {dW_1(X)}{dX} - Y^2  \frac {d^2 W_2(X)}{dX^2}
  \right) =0.
\end{split}
\end{equation*} 
Solving the above equation, we obtain
\begin{equation*}
\begin{split}
F_0(X,Y,Z)=& H_0(X,Z) Y^{k_3}+\left( \frac{1}{2} Y^2 \frac {d^2 W_2(X)}{dX^2}  -2(b-cX)YW_2(X) - Y\frac {d W_1(X)}{dX}  \right) Y^{k_3} \\
+&\ln (Y) \left( -aX\frac {dW_2(X)}{dX} - (a k_3 -2X^2)W_2(X) \right) Y^{k_3} 
-ak_3\left( \frac{a (k_3-1)X^2 W_2(X) }{2Y^2}
-\frac{X W_1(X)}{Y}  \right)Y^{k_3},
\end{split}
\end{equation*}
where $H_{0}(X,Z)$ is a polynomial of the variables $X$ and $Z$.
Since $F_{0}(X,Y,Z)$ is a polynomial, this required that $W_2(X)=0$, which implies that
\begin{equation}\label{F1}
F_1(X,Y,Z)=W_1(X)  Z Y^{k_3},
\end{equation}
\begin{equation}
F_0(X,Y,Z)=-\frac{d W_1(X)}{d X}Y^{k_3+1} - a k_3 X W_1(X) Y^{k_3-1} + H_0(X, Z) Y^{k_3}.
\end{equation}
Since $F_0$ is a polynomial of weight degree $0$, we can assume that $H_{0}(X,Z)=W_{0}(X)$, which is a polynomial for $X$ only. Lastly, computing the coefficient of $\lambda^{n+1}$ in (\ref{w2}), we obtain
\begin{equation}\label{np1}
\begin{split}
&\left( (-ak_3-X^2)+(b-cX)Y \right) W_1(X) Y^{k_3} + \left( aX \frac {dW_1(X)}{dX}  - Y^2 \frac {d^2W_1(X)}{dX^2} +Y \frac {dW_0(X)}{dX}   \right) Y^{k_3} \\
&+\left(  \frac{a^2{k_3}({k_3} - 1)  X^2 W_1(X)}{Y^2}  - \frac{ak_3XW_0(X)}{Y} \right) Y^{k_3}  =0.
\end{split}
\end{equation}
The above equation is a polynomial of variables $X$ and $Y$, computing the coefficients in (\ref{np1}) of terms $Y^{k_3+2},Y^{k_3+1}, Y^{k_3}$ and $Y^{k_3-1}$ respectively, we obtain
\begin{equation}\label{Yp2}
	-\frac{d^2}{dX^2}W_1(X)=0,
\end{equation}
\begin{equation}\label{Yp1}
	\frac{d}{dX}W_0(X)+(b-cX)W_1(X)=0,
\end{equation}
\begin{equation}\label{Yp0}
	aX\frac{d}{dX}W_1(X)-(X^2 +k_3)W_1(X)=0,
\end{equation}
\begin{equation}\label{Ym1}
	-a k_3 X W_0(X)=0.
\end{equation}
Solving the equation (\ref{Yp2}) for $W_1(X)$, we can write $W_1(X)=d_1 X+d_0$, where $d_1$ and $d_0$ are arbitrary constant. Substituting $W_1(X)=d_1 X+d_0$ in equation (\ref{Yp0}), we obtain
\begin{equation}
-d_1 X^3 -d_0 X^2 + ad_1(1-k_3)X-ad_0k_3=0.
\end{equation}
One possibility that the above polynomial becomes zero is that $d_1=d_0=0$. This implies that $W_1(X)=0$, $F_1(X,Y,Z)=0$ and $F_0(X,Y,Z)=W_0(X) Y^{k_3}$. Substituting $W_1(X)=0$ in the equation (\ref{Yp1}) and solve it for $W_0(X)$, we obtain $W_0(X)=C_0$, where $C_0$ is a constant. This gives that from equation (\ref{Ym1}), $-a C_0 k_3 X=0$. Since $k_3>0$ and $C_0\ne0$, then $a=0$, thus, $F_0(X,Z,Y)=C_0 Y^{k_3}$. 

To sum up, From $F(\lambda,X,Y,Z)=\sum_{i=0}^{n}\lambda^i F_{n-i}(X,Y,Z)$, we have $F= C_0 Y^{k_3}$. This conclude that, $f=y^{k_3}$ is Darboux polynomial with cofactor $K=k_3 z$. This completes the proof of Theorem \ref{INT-S1-big} (2).
\end{proof}

\begin{proof}[Proof of Theorem \ref{INT-S1-big} (3)]
	The proof comes directly from Theorem \ref{INT-S1-big} (1) and (2) with Lemma \ref{THration1}.
\end{proof}

\begin{proof}[Proof of Theorem \ref{INT-S1-big} (4)]
Let $E=e^{\frac{g}{h}}$ be an exponential factor of the system (\ref{1ybig}) with cofactor $L=l_0+l_1 x+l_2 y+l_3 z$, where $g,h \in \mathbb{C}[x,y,z]$ with $g$ and $h$ are relatively prime and $l_i\in \mathbb{C}$ for $i=0,1,2,3$. To give the complete proof, we consider two cases;

\textbf{Case (1):} If $a\ne 0$, then from Theorem \ref{INT-S1-big}(2) and Proposition \ref{propo exp}, $h$ is a constant (let say $h=1$). Thus, $E=e^{g}$, then we have
\begin{equation}\label{e1}
y \frac{\partial e^g}{\partial x}+(-ax+yz)  \frac{\partial e^g}{\partial y}+(by-cxy-x^2 )  \frac{\partial e^g}{\partial z}=Le^g .
\end{equation}
Simplifying
\begin{equation}\label{e2}
y \frac{\partial g}{\partial x}+(-ax+yz)  \frac{\partial g}{\partial y}+(by-cxy-x^2 )  \frac{\partial g}{\partial z}=L, 
\end{equation}
Let $g$ be $g(x,y,z)=\sum_{i=0}^n g_i (x,y)  z^i $, where each $g_i$ is a polynomial of the variables $x$ and $y$. Firstly, we consider $n\ge2$.

Now, computing the coefficient of $z^{n+1}$ in (\ref{e2}), we obtain
$y \frac{\partial g_n (x,y)}{\partial y}=0$. That is $g_n(x,y)=G_n (x)$, where $G_n(x)$ is a polynomial of the variable $x$. Computing also the coefficient of $z^n$ in (\ref{e2}), we obtain
$$y \frac{\partial g_{n-1}(x,y)}{\partial y}+y \frac{dG_n(x)}{dx}=0.$$
That is $g_{n-1} (x,y)=-y  \frac{dG_n(x)}{dx}+G_{n-1}(x)$, 
where $G_{n-1}(x)$ is an arbitrary polynomial of the variable $x$. Next, computing the coefficient in (\ref{e2}) of $z^{n-1}$, we obtain
\begin{equation*}
y \frac{\partial g_{n-2} (x,y)}{\partial y}+ax \frac{dG_n(x)}{dx}+ y \frac{dG_{n-1}(x)}{dx} - y^2 \frac{d^2G_n(x)}{dx^2}  +n(b-cx)yG_n (x)-nx^2 G_n (x) =0. 
\end{equation*}
We can solve the above equation for $g_{n-2} (x,y)$, we obtain
\begin{equation*}
g_{n-2}(x,y)= \frac{1}{2} \,y^2 \frac{d^2G_n(x)}{dx^2} - y \frac{dG_{n-1}(x)}{dx} +(-b+cx)nyG_n(x)
+\ln(y) \left( nx^2 G_n (x) -ax \frac{dG_n(x)}{dx} \right) +G_{n-2}(x),
\end{equation*}
where $G_{n-2}(x)$ is a polynomial of the variable $x$. Since $g_{n-2} (x,y)$ is a polynomial and $a\neq 0$, it is required that $G_n (x)=0$. This implies that $g_n=0$ for $n\geq 2$. Hence, we have $g(x,y,z)=g_0 (x,y)+g_1 (x,y)z$. The equation (\ref{e2}) becomes
\begin{equation}\label{e3}
(yz^2-axz) \frac{\partial g_1(x,y)}{\partial y} +yz \frac{\partial g_1(x,y)}{\partial x} +(yz-ax) \frac{\partial g_0(x,y)}{\partial y} +y \frac{\partial g_0(x,y)}{\partial x}  +(by-x^2-cxy)g_1(x,y)=L.
\end{equation}
Compute the coefficients in (\ref{e3}) of $z^2$,$z^1$ and $z^0$ respectively, we obtain the following differential equations
\begin{equation}\label{e31}
y\frac{\partial g_1(x,y)}{\partial y}=0,
\end{equation}
\begin{equation}\label{e32}
-ax\frac{\partial g_1(x,y)}{\partial y} + y\frac{\partial g_0(x,y)}{\partial y} + y\frac{\partial g_1(x,y)}{\partial x}-l_3=0,
\end{equation}
\begin{equation}\label{e33}
-ax \frac{\partial g_0(x,y)}{\partial y} +y \frac{\partial g_0(x,y)}{\partial x} +(by-x^2-cxy)g_1(x,y)-l_0-l_1 x-l_2 y=0.
\end{equation}
Solve the equation (\ref{e31}) for $g_1 (x,y)$, we obtain $g_1 (x,y)=G_1 (x)$, where $G_1 (x)$ is a polynomial of the variable $x$. Substituting $g_1=G_1(x)$ in (\ref{e32}), we obtain
\begin{equation*}
g_0(x,y)=- y\frac{d G_1(x)}{d x}+l_3 \ln(y)+G_0(x),
\end{equation*}
where $G_0 (x)$ is a polynomial of the variable $x$. Since $g_0 (x,y)$ is a polynomial, then $l_3=0$.

Therfore, equation (\ref{e33}) can be simplified into
\begin{equation*}
y \frac{d G_0(x)}{d x} +\left( (b-cx)G_1(x)-l_2 \right) y-y^2 \frac{d^2 G_1(x)}{d x^2} + \left( ax \frac{d G_1(x)}{d x} -x^2G_1(x)-l_0-l_1 x \right) =0.
\end{equation*}
Computing the coefficients of $y^i$ for $i=2,1,0$ respectively, we can derive the following:
\begin{equation}\label{ex31}
-\frac{d^2 G_1(x)}{dx^2}=0,
\end{equation}
\begin{equation}\label{ex32}
\frac{d G_0(x)}{dx}+(b-cx)G_1(x)-l_2=0,
\end{equation}
\begin{equation}\label{ex33}
a x \frac{d G_1(x)}{dx} -x^2 G_1(x)-l_1 x-l_0=0.
\end{equation}
Solve the equation (\ref{ex31}) for $G_1(x)$, we obtain $G_1(x)=a_1x+a_0$, where $a_1,a_0$ are arbitrary constant. Then, the equation (\ref{ex33}) becomes
$$-a_1 x^3+ a_0 x^2 +(a a_1 -l_1)x-l_0=0.$$
Since $a\ne 0$, it is required that $a_1=a_0=l_1=l_0=0$. Now, solving equation (\ref{ex32}) for $G_0$, we obtain $G_0(x)=l_2 x+c_0$, where $c_0$ is constant. As a result, we have $g=g_0=l_2 x+c_0$, that is $e^{l_2 x+c_0}$ is an exponential factor with cofactor $L=l_2 y$.

\textbf{Case (2):} When $a=0$, then according to Proposition \ref{propo exp} and Theorem \ref{INT-S1-big} (2), the exponential factors of system (\ref{1ybig}) can be expressed as $E=e^{\frac{g}{y^m}}$ for some non-negative integer $m$,  where $g\in \mathbb{C}[x, y, z]$, in which $g$ and $y^m$ are relatively prime. By definition of the exponential factor, directly, we have 
\begin{equation}\label{ex2}
	y \frac{\partial g}{\partial x}+yz  \frac{\partial g}{\partial y}+(by-cxy-x^2)  \frac{\partial g}{\partial z}-mzg=L y^m .
\end{equation}
Now, we consider two cases;

\textbf{Case I:} For $m \ge 1$, the restriction of $g$ to $y=0$ is denoted as $\acute{g}$ is the polynomial, defined by $g(x,y,z)| _{y=0} =\acute{g}$. Then equation (\ref{ex2}), becomes
\begin{equation}\label{ex3}
-{x}^{2}{\frac {\partial \acute{g}}{\partial z}} -  mz\acute{g}=0.
\end{equation}
Let $\acute{g}(x,z)=\sum_{i=0}^{n}\acute{g}_i(x,z)$, where each $\acute{g}_i$ is a homogeneous polynomial of degree $i$ of the variables $x$ and $z$. 

By computing the terms of degree $n+1$ from (\ref{ex3}), we obtain
\begin{equation}
-{x}^2{\frac {\partial \acute{g}_{n}(x,z)}{\partial z}} -mz\acute{g}_{n}(x,z) =0.
\end{equation}
Solving the above equation for $\acute{g}_n$, we obtain
\begin{equation}
\acute{g}_{n}(x,z) =G_n(x) e^{-\frac{mz^2}{2x^2}},
\end{equation}
Since $\acute{g}_{n}(x,z)$ is a polynomial and $m\ge 1$, this gives that $G_n(x)=0$. Then , $\acute{g}_{i}(x,z) = 0$ for each $i=0,\ldots,n$, thus $\acute{g}(x,z) = 0$. This case can not be taken.
	
\textbf{Case II:} For $m=0$, we have $E=e^g$, where $g\in \mathbb{C}[x, y, z]$. Setting $a = 0$, in Theorem \ref{INT-S1-big} (4) Case (1), We obtain $e^{g(x,y,z)}=e^{l_2 x + c_0}$ with the cofactor $L = l_2 y$. 
This completes the proof of Theorem \ref{INT-S1-big} (4).
\end{proof}

\begin{proof}[Proof of Theorem \ref{INT-S1-big} (5)]
According to Theorem \ref{thDarType}, system (\ref{1ybig}) has a first integral of the Darboux type if and only if there are $\lambda_i$ and $\mu_j$, which are not all zero, satisfying equation (\ref{EqDarType}). By Theorems \ref{INT-S1-big} (2) and (4), we can consider the following cases;
\begin{enumerate}
\item If $a\ne 0$, then system (\ref{1ybig}) has no Darboux polynomials. By Theorem \ref{INT-S1-big} (4), there is only one exponential factor $e^x$ with cofactor $L=y$. Hence, the equation (\ref{EqDarType}) becomes $\mu_1y = 0$. This gives that $\mu_1= 0$.
\item If $a = 0$, then system (\ref{1ybig}) has one Darboux polynomial $f=y$ with cofactor $K=z$. By Theorem \ref{INT-S1-big} (4), there is only one exponential factor $e^x$ with cofactor $L=y$. Hence, the equation (\ref{EqDarType}) becomes $\lambda_1 z + \mu_1 y = 0$. This gives that $\lambda_1 = \mu_1 = 0$.
\end{enumerate}
This completes the proof of Theorem \ref{INT-S1-big} (5).
\end{proof}

\begin{proof}[Proof of Theorem \ref{THformal}]
	System (\ref{1ybig}) has a line of equilibrium points formed by $E(0, 0, z_0)$, where $z_0 \in \mathbb{R}$.
	The characteristic polynomial of the Jacobian matrix at $E(0,0, z_0)$ is given by $\lambda^3 - z_0 \lambda^2 + a \lambda = 0$.  Hence, the eigenvalues are	$\lambda_1=0,  \lambda_{2,3}=\frac{z_0 \pm \sqrt{z_0^2 -4a}}{2} $.

We consider following cases:
\begin{enumerate}
	\item If \textbf{$a>0$}, then $\lambda_2 \lambda_3 = a$ and $\lambda_2^2 = \left( \frac{z_0 + \sqrt{z_0^2 -4a}}{2}\right) ^2 >0$, for $z_0^2 \ge 4a$. Hence, $k_2 \lambda_2 + k_3 \lambda_3 = \frac{1}{\lambda_2}(k_2 \lambda_2^2 + k_3\lambda_2\lambda_3) \ne 0$, for all $k_2, k_3\in \mathbb{Z}\cup\{0\}$ with $k_2 + k_3 >0$. By Theorem \ref{Theoana1}, system (\ref{1ybig}) has a formal first integral in a neighborhood of $(0, 0, z_0)$ except the origin.
	
	On the other hand, if $z_0^2<4a$ i.e., $-2\sqrt{a}<z_0< 2\sqrt{a}$, then $\lambda_{2,3}=\frac{z_0\pm i\sqrt{4a-z_0^2}}{2}$. Hence, either all eigenvalues have positive real parts or all have negative real parts. By Theorem \ref{Theoana2}, system (\ref{1ybig}) has an analytic first integral in a neighborhood of the equilibrium points $(0, 0, z_0)$ for $z_0 \in (-2\sqrt{a}, 2\sqrt{a})\setminus \{0\}$.
		
	\item If \textbf{$a<0$}, then $k_2 \lambda_2 + k_3 \lambda_3 =(k_2 - k_3) \frac{\sqrt{z_0^2 -4a}}{2} +(k_2 +k_3) \frac{z_0}{2} \ne 0$, for all non-negative integers $k_2$ and $k_3$, such that $k_2+k_3 > 0$ and $z_0\ne 0$.
	Otherwise, if $(k_2 - k_3) \frac{\sqrt{z_0^2 -4a}}{2} +(k_2 +k_3) \frac{z_0}{2} = 0$, then $$\frac{k_3-k_2}{k_3+k_2}=\frac{z_0}{\sqrt{z_0^2-4a}}.$$ 
	Which is impossible, because $z_0\in \mathbb{R}\setminus \{0\}$ and $z_0^2-4a >0$, we can pick some value for $z_0$ in which the right side of the aforementioned equation becomes irrational. Consequently, by Theorem (\ref{Theoana1}) and the equilibrium points $(0,0,z_0)$ being non-isolated, then system (\ref{1ybig}) has a formal first integral in a neighborhood of $(0, 0, z_0)$ except the origin.
	\end{enumerate}
This conclude the proof of Theorem \ref{THformal}.	
\end{proof}

\section*{Acknowledgments}
We want to thank the referees for their insightful comments and helpful recommendations to enhance the way this work was presented.

\section*{Appendix: Averaging theory of first order for limit cycles}
Now we'll go through the basic averaging theory for Lipschitz differential systems  which we'll need to prove result of isolated limit cycle bifurcate from zero-Hopf point. The following theorem offers a first-order of the averaging theory for differential system which founded in [\cite{buicua2004averaging},\cite{buicua2008yu}] and used in  [\cite{kassa2021limit},\cite{llibre2016zero},\cite{diab2021zero}]. For more information and the proof see previous references.
\begin{theorem} \label{AvThe}
	Consider the differential equation
	\begin{equation}\label{Av1}
		\dot{\mathbf{x}}=\epsilon F_1(t,\mathbf{x})+\epsilon^2 F_2(t,\mathbf{x}) ,
	\end{equation}
	where $F_1:\mathbb{R} \times \Omega \to \mathbb{R}^n$, $F_2:\mathbb{R}\times\Omega \times \left(0,\epsilon_0 \right] \to \mathbb{R}^n $ are continuous functions and $2\pi$-periodic in $t$, where $\Omega \subseteq \mathbb{R}^n $ open. We set $F_{10}:\Omega \to \mathbb{R}^n$ and define
	\begin{equation}\label{Av2}
		F_{10}(\mathbf{x})=\frac{1}{2\pi}\int_{0}^{2\pi} F_1(t,\mathbf{x}) \, dt
	\end{equation}
	and assume that, 
	\begin{enumerate}
		\item $F_1$ and $F_2$ are locally Lipschitz in $\mathbf{x}$.
		\item Let $F_{10}$ be $C^1$ function, for $\mathbf{s}\in D$ with $F_{10}(\mathbf{s})=0$, If the determinant $\left( det(DF_{10}(\mathbf{s})) \ne 0\right) $ of the Jacobian matrix of function $F_{10}$ at $\mathbf{s}$ is not zero, then there exists a neighborhood $U$ of $\mathbf{s}$ such that $F_{10}(\mathbf{z})\ne 0, \forall \mathbf{z}\in V-\{s\}$ and Brouwer degree of $F_{10}$ in the neighborhood $V$, $d_B(F_{10},V,\mathbf{s},0)\in \{-1,1\}$.
	\end{enumerate}
	Then for $|\epsilon|> 0$ sufficiently small, there exists an isolated $2\pi-$periodic solution $\mathbf{x}(t, \epsilon)$ of system (\ref{Av1}) such that $\mathbf{x}(0, \epsilon) \to \mathbf{s}$ as $\epsilon \to 0$. Moreovere, if all eigenvalues of the matrix $DF_{10}(\mathbf{s})$ have negative real parts, then the limit cycle $\mathbf{x}(t, \epsilon)$ is stable. If
	some of the eigenvalue has positive real part the limit cycle $\mathbf{x}(t, \epsilon)$ is unstable.
\end{theorem}

\bigskip
\bibliographystyle{apalike3}
\bibliography{bibliography.bib}

\end{document}